\documentclass[11pt]{amsart}
\usepackage{float}
\usepackage{epsfig}
\usepackage{graphicx}
\usepackage{caption}
\usepackage{subcaption}
\usepackage{color}

\title{Random walk with heterogeneous sojourn time}% with location dependent travelling time}

\author{Jaywan Chung}
\address[Jaywan Chung]{Energy Conversion Research Center, Korea Electrotechnology Research Institute, 12, Jeongiui-gil, Changwon-si, Gyeongsangnam-do, 51543, Korea}
\email{jchung@keri.re.kr}

\author{Yong-Jung Kim}
\address[Yong-Jung Kim]{ Department of Mathematical Sciences, KAIST,e 291 Daehak-ro, Yuseong-gu, Daejeon, 34141, Korea}
\email{yongkim@kaist.edu}

\author{Min-Gi Lee}
\address[Min-Gi Lee]{ Department of Mathematics, Kyungpook National University, 80 Daehak-ro, Buk-gu, Daegu, 41566, Korea}
\email{leem@knu.ac.kr}

\def\eps{\epsilon}
\def\tri{\triangle}
\def\R{\mathbb{R}}
\def\Z{\mathbb{Z}}
\def\erf{\mathrm{erf}}
\def\erfc{\mathrm{erfc}}

\newtheorem{theorem}{Theorem}[section]
\newtheorem{lemma}{Lemma}[section]
\newtheorem{corollary}{Corollary}[section]
\newtheorem{proposition}{Proposition}[section]

\newtheorem{definition}{Definition}[section]

\numberwithin{equation}{section}

\begin{document}
\maketitle
\begin{abstract}
We introduce a discrete-time random walk model on a one-dimensional lattice with a nonconstant sojourn time and prove that the discrete density converges to a solution of a continuum diffusion equation.
Our random walk model is not Markovian due to the heterogeneity in the sojourn time, in contrast to a random walk model with a nonconstant walk length.
We derive a Markovian process by choosing appropriate subindexes of the time-space grid points, and then show the convergence of its discrete density through the parabolic-scale limit. We also find the Green's function of the continuum diffusion equation and present three Monte Carlo simulations to validate the random walk model and the diffusion equation.
\end{abstract}

\section{Introduction}

The purpose of this paper is to introduce a random walk model with a nonconstant sojourn time and to find the corresponding diffusion equation. We will see that the heterogeneity in the sojourn time behaves differently from the one in the walk length, and thus the diffusivity coefficient alone cannot explain a heterogeneous diffusion phenomenon. To see this, we consider a position-jump model, where the particles jump the distance of the walk length $\tri x$ with a random direction every sojourn time $\tri t$. The random walk model with parameters,
\begin{equation}\label{1.1}
\tri x=1\quad\text{and}\quad \tri t=\tau(x):=
\begin{cases}
1,\quad x<0,\\
2,\quad x\ge0,
\end{cases}
\end{equation}
is one of the simplest models with a nonconstant sojourn time. In this situation, the particles in the region $x<0$ walk twice while those in the region $x>0$ walk once. This simplest case can serve as a building block to construct a general case. In this setup, the heterogeneity in the walk length $\tri x$ is forgotten, and we focus on the heterogeneity in the sojourn time $\tri t$. We will work with a discrete-time model on a lattice corresponding to \eqref{1.1} and study the corresponding discrete stochastic process.

% since a continuous-time model cannot handle the nonconstant sojourn time properly.

Einstein's random walk model \cite{011} for the Brownian motion is a discrete-time model with a constant sojourn time $\tri t$. The walk length follows a normal distribution with a constant deviation which corresponds to $\tri x$ of our case. Then, the diffusivity is given by
\begin{equation}\label{D}
D=\frac{(\tri x)^2}{2n\tri t},
\end{equation}
where $n$ is the spatial dimension. If $\tri x$ and $\tri t$ are of microscopic scale, the particle density $u(t,x)$ of macroscopic-scale variables satisfies the diffusion equation
\begin{equation}\label{HeatE}
u_t=D\Delta u,
\end{equation}
where $\Delta u:=\sum_{j=1}^n\partial^2_{x_i} u$ is the Laplace operator. More precisely, if we take the diffusion limit as $\tri x \to 0$ while keeping the ratio in \eqref{D} fixed, the particle density satisfies the diffusion equation \eqref{HeatE}.

After Einstein's random walk theory, a huge systematic development was made and much of it was based on Smoluchowski's independent work \cite{012}. On the other hand, numerous attempts have been made to use random walk theory to explain the nonconstant steady states of the thermal diffusion phenomenon observed in 1856 by Ludwig \cite{016} and later by Soret \cite{017}. Non-constant steady states are also observed in many other instances of physical and biological problems, and heterogeneous diffusion models have been developed to explain them (see \cite{009,018,019}). If the environment is heterogeneous due to nonconstant temperature, heterogeneous media, geometric structure, etc., we can incorporate the heterogeneity into the random walk model by assuming a spatially varying walk length and sojourn time. Let
\begin{equation}\label{heterogeneity}
\tri x\big|_{\text{at } x}=\ell(x)\quad\text{and}\quad \tri t\big|_{\text{at } x}=\tau(x).
\end{equation}
Then, the relation in \eqref{D} still gives the heterogeneous diffusivity, 
\begin{equation}\label{D(x)}
D(x)=\frac{\ell^2(x)}{2n\tau(x)}.
\end{equation}
However, the diffusion equation \eqref{HeatE} is no longer valid. Among the infinitely many possibilities, three are well known:
\begin{eqnarray}
\label{Fick} u_t &=& \nabla\cdot(D(x)\nabla u), \\
\label{Wereide} u_t &=& \nabla\cdot(\sqrt{D(x)}\nabla(\sqrt{D(x)} u)), \\
\label{Chapman} u_t &=& \Delta (D(x)u),
\end{eqnarray}
which are called Fick \cite{001}, Wereide \cite{002}, and Chapman \cite{003}, respectively. If the diffusivity $D$ is constant, the three equations are identical to \eqref{HeatE}. If not, they are all different. Furthermore, none of the three diffusion laws above are universal laws, because as we will see below, the heterogeneities in $\tri t$ and $\tri x$ behave differently, making it impossible to explain a diffusion phenomenon with diffusivity alone.

Note that the heterogeneity of the parameters given in \eqref{heterogeneity} does not determine a diffusion equation. It is determined by the way the spatial heterogeneity of the parameters is taken into account. For example, suppose that a particle jumps from a departure point $x$ to an arrival point $y$. Consider the case where the walk length is given by
\begin{equation}\label{reference_x}
\tri x:=|y-x|=\ell(ay+(1-a)x),\quad 0\le a\le 1.
\end{equation}
If $a=0$, \eqref{reference_x} indicates that the heterogeneity is taken from the departure point $x$. Similarly, if $a=1$, the heterogeneity is taken from the arrival point $y$. We can take any \emph{reference point} between $x$ and $y$ by choosing $a\in[0,1]$. Alfaro et al. \cite{004} formally derived a diffusion equation
\begin{equation}\label{General}
u_t=\nabla\cdot(D^q\nabla(D^{1-q} u))
\end{equation}
in the context of non-local diffusion with $q=2a$. Kim and Lim \cite{005,006} showed in a different setting that the particle density converges to the solution of \eqref{General} when $q=a$. The three cases of \eqref{Fick}--\eqref{Chapman} are special cases of \eqref{General} when $q=1,0.5$, and $0$, respectively. These results are obtained under the assumption that the sojourn time is constant, and we can ask the same question under a heterogeneous sojourn time $\tri t$.

Langevin \cite{013} obtained the same result as Einstein in a simpler way by introducing a stochastic differential equation, now called the Langevin equation. His idea has been developed as a fundamental tool for the analysis of random phenomena (see \cite{014,015}). In a heterogeneous environment, the Langevin equation becomes a nonlinear stochastic differential equation,
\begin{equation}\label{SDE}
\dot x(t) = k(x) + \ell(x)W(t),
\end{equation}
where $x(t)$ is the position of a particle and $W(t)$ is a stationary stochastic process. The function $k(x)$ controls a heterogeneous advection phenomenon and is zero in our context. The function $\ell(x)$ is the deviation when $W(t)$ is the Gaussian white noise and plays the role of the walk length in our model. The diffusivity is given by the same formula \eqref{D(x)} with $\tau(x)\equiv1$. The equation \eqref{SDE} requires an appropriate interpretation of taking $x$ in $\ell(x)$. It\^o's interpretation is to take $x$ before the pulse and the probability density function satisfies \eqref{Chapman}. Stratonovich's interpretation is to take the midpoint between the two points before and after the pulse, and \eqref{Wereide} is satisfied. In other words, the diffusion equation for the stochastic differential equation \eqref{SDE} is also determined by the choice of the reference point in the same way (see \cite{022,021} for further discussion). However, there is no component in the Langevin equation \eqref{SDE} that corresponds to the heterogeneity in the sojourn time $\tri t$.

To study the heterogeneity in the sojourn time $\tri t$, we consider a one-dimensional case where the walk length is a constant $\ell(x)\equiv1$ and the sojourn time $\tau(x)$ is a step function:
\begin{equation}\label{1.13}
\tri x=\eps,\qquad
\tri t=\eps^2\tau(x),\quad \tau(x)=
\begin{cases}
1 & \text{if $x<0$,}\\
2 & \text{if $x>0$},
\end{cases}
\end{equation}
where the small parameter $\eps>0$ is introduced to take the diffusion limit as $\eps\to0$. We will show that, as $\eps\to0$, the particle density converges to a solution of
\begin{equation}\label{ux1}
u_t=\frac{1}{2}\Big({u\over \tau(x)}\Big)_{xx},\quad u(0,x)=u_0(x),
\end{equation}
where $u_0$ is the initial value.

The heterogeneity in the sojourn time plays a different role than in the walk length; the corresponding diffusion equation is independent of the reference point. To be precise, even if we take the sojourn time as 
\begin{equation}\label{reference_t}
\tri t=\tau(by+(1-b)x),\quad 0\le b\le 1,
\end{equation}
the corresponding diffusion equation is independent of the parameter $b\in[0,1]$ (see \cite[Lemma 5.1]{006} and the Monte Carlo simulations in Figure \ref{fig4}). Therefore, it is sufficient to consider just one case to cover them all. 

To our knowledge, the diffusion limit of a random walk system as a discrete stochastic process with a nonconstant sojourn time has never been obtained. For continuous-time cases, there is a Montroll-Weiss theory of continuous-time random walk (CTRW) \cite{025,026} and we will explain briefly the theory below. If $\tri t$ is not spatially constant, the resulting random walk system is not Markovian  and most of tools for a Markov chain are not applicable. What is done so far is to interpret the departing rate as the reciprocal of the sojourn time, i.e.,
\begin{equation}\label{gamma}
\gamma(x):=\frac{\tau_0}{\tau(x)},\quad 0<\tau_0\le\inf_x \tau(x).
\end{equation}
Then, one may take a discrete-time model
\[
u(t+\tau_0,x)=\frac{\gamma(x-\tri x)}{2} u(t,x-\tri x)+\frac{\gamma(x+\tri x)}{2} u(t,x+\tri x) +(1-\gamma(x))u(t,x).
\] 
If the reference point of the departing rate $\gamma$ is the departure point as above, the diffusion equation turns into
\[
u_t=\frac{1}{2\tau_0}(\gamma(x)u)_{xx},
\]
which is equivalent to \eqref{ux1} under the relation \eqref{gamma}. If not, the resulting diffusion equation is not equivalent to \eqref{ux1}. The departing rate can be treated in both continuous-time and discrete-time models (see \cite{005,006}). The concept of the departing rate has been widely used without linking it to the sojourn time (see \cite{008,020,007}). Note that in regard to \eqref{gamma} our main result supports the idea that the reference point of the departing rate should be the departure point, i.e., of the It\^o type.

The continuous-time random walk model of Montroll and Weiss \cite{025} also has been widely studied and has contributed to understanding phenomena of chemical kinetics in crystalline solids or  porous media \cite{025,026,024,027}. In this CTRW theory, the sojourn time of a particle that have just made a jump at time $t$ to arrive at $x$ is a random variable. By imposing the probability density function of the random variable dependent on $x$ \cite{026}, the heterogeneity can be entailed in the model. A CTRW process corresponds to its so called {\it master equation}, which is better known as the Kolmogorov equation in mathematics community. The prescribed probability density function of the sojourn time characterizes the equation and it can be in widely general form, including non-local ones.

The rest of the paper is organized as follows. In Section \ref{sect.2}, a discrete-time random walk model with the parameters in \eqref{1.13} is introduced. Since the sojourn time is heterogeneous, the obtained random walk system is not Markovian. The main effort of this section is to transform it into a Markovian. In Section \ref{sect.3}, we construct a Lipschitz continuous interpolation of the discrete solution of the random walk problem, and find uniform estimates using difference quotients. These estimates give a convergent subsequence of the interpolation. The convergence to the weak solution is finally proved in Section \ref{sect.4}. In Section \ref{sect.5}, the Green's function for the diffusion equation \eqref{ux1} is obtained explicitly. It is compared with the discrete random walk system. Three Monte Carlo simulations are given in Section \ref{sect.6} showing the behavior of the Green's function and steady states. In addition, it is illustrated that the steady states are independent of the reference point in the sojourn time. The computation codes for these numerical simulations are given in the Appendix.

\section{From non-Markovian to Markovian}\label{sect.2}

We consider a one-dimensional random walk system where a particle jumps to one of the two adjacent grid points with the equal probability ${1\over2}$. We assume that the grid points are equally spaced with a mesh size $\tri x=\eps$ and that the sojourn time $\tri t$ is constant in each of the two divided regions, $x<0$ and $x>0$. We take
\begin{equation}\label{tau(x)}
\tri x=\eps,\quad
\tri t=\eps^2\tau(x),\quad
 \tau(x) =
\left\{\begin{array}{ll}
        1 & \text{if } x<0,\\
        2 & \text{if } x>0.
        \end{array}\right.
\end{equation}
Note that the step function $\tau(x)$ is to denote the spatial heterogeneity in $\tri t$, and $\eps$ and $\eps^2$ are to denote the microscopic-scale dimensions of the walk length and the sojourn time, respectively. The diffusion limit will be taken as $\eps\to0$. In this case, the diffusivity is fixed as
\[
D(x)=\frac{(\tri x)^2}{2\tri t}=\frac{1}{2\tau(x)}.
\]
The convergence proof of the paper depends on the fact that the sojourn time in one region is exactly twice long as in the other. 

The random walk motion can be interpreted in two ways. The first way is to consider $\tri t$ as the sojourn time, which means that a particle stays at a position for the time period $\tri t$ and then immediately jumps to the next position. We may take the heterogeneity in the sojourn time as
\[
\tri t=\tau(by+(1-b)x),\quad 0\le b\le 1,
\]
where $x$ is the departure point and $y$ is the arrival point. If $b=0$, the sojourn time is independent of the arrival point and if $b\ne0$, it is not. However, since the resulting diffusion limit is independent of the choice of $b$ (see \cite[Lemma 5.1]{006}), we may choose any.

The second way is to consider $\tri t$ as the travel time, which means that a particle departs for the next position immediately after arriving and then arrives at the next position after the travel time $\tri t$. The concept of travel time is more restrictive, since the travel time from a position $x$ to a position $y$ should be equal to the travel time from $y$ to $x$. This property holds if the reference point of the sojourn time is taken as the midpoint $\frac{x+y}{2}$, i.e., $b=0.5$ or
\begin{equation}\label{b=0.5}
\tri t=\tau((y+x)/2).
\end{equation}
This choice is made for our convenience in the following recursive relations; a different choice may result in different relations and difficulties. However, the resulting diffusion equation should be the same.

We denote by $p_n^j$ the probability that a particle is located at $x=x^j$ at time $t=t_n$. We denote the space discretization by superscripts and the time discretization by subscripts. The time-space grid points used in this paper are
$$
(t_n,x^j):=(n\eps^2,j\eps),\qquad n\in\Z_+,\ j\in\Z,
$$
where the time step size $\eps^2$ and the space mesh width $\eps$ are small real numbers, and $n$ and $j$ are integers. Then, due to the heterogeneity of the travel time $\tri t$ in \eqref{b=0.5}, the probability $p_n^j$ is given by averaging two values at the two adjacent grid points at two moments, $t_{n-1}$ or $t_{n-2}$, depending on $j$, which is
\begin{equation} \label{RW1}%\tag{RW1}
    p^j_n =\frac{1}{2}\times
        \begin{cases}
        p^{j-1} _{n-1} + p ^{j+1} _{n-1} & \text{if } j<0,\\
        p^{-1} _{n-1} + p ^{1} _{n-2} & \text{if } j=0,\\
        p^{j-1} _{n-2} + p ^{j+1} _{n-2} & \text{if } j>0.
        \end{cases}
\end{equation}
Note that this process is not a Markov chain, since $p_n$ is determined by two previous steps, $p_{n-1}$ and $p_{n-2}$. By applying these relations twice, we obtain
\[
p_n^j=\frac{1}{4}\times
\begin{cases}
p_{n-2}^{j-2}+2p_{n-2}^{j}+p_{n-2}^{j+2} & \mbox{if } j<-1,\\
p_{n-2}^{-3}+2p_{n-2}^{-1}+p_{n-3}^1 & \mbox{if } j=-1,\\
p_{n-2}^{-2}+p_{n-2}^{0}+p_{n-4}^{0}+p_{n-4}^{2} & \mbox{if } j=0,\\
p_{n-3}^{-1}+2p_{n-4}^{1}+p_{n-4}^3 & \mbox{if }j=1,\\
p_{n-4}^{j-2}+2p_{n-4}^{j}+p_{n-4}^{j+2} & \mbox{if }j>1.
\end{cases}
\]
The next step is to rewrite the relation with elements in the $(n-2)$th step only. Since $p_{n-4}^{0}+p_{n-4}^{2}=2p_{n-2}^1$, $p_{n-3}^{-1}+2p_{n-4}^{1}+p_{n-4}^3=2p_{n-2}^0+2p_{n-2}^2$, and $p_{n-4}^{j-2}+2p_{n-4}^{j}+p_{n-4}^{j+2}=2p_{n-2}^{j-1}+2p_{n-2}^{j+1}$,
the above is written as
\begin{equation} \label{2.4}
p_n^j=\frac{1}{4}\times
\begin{cases}
p_{n-2}^{j-2}+2p_{n-2}^{j}+p_{n-2}^{j+2}, & \mbox{if } j<-1,\\
p_{n-2}^{-3}+2p_{n-2}^{-1}+\boxed{p_{n-3}^1}, & \mbox{if } j=-1,\\
p_{n-2}^{-2}+p_{n-2}^{0}+2p_{n-2}^1, & \mbox{if } j=0,\\
2p_{n-2}^0+2p_{n-2}^2, & \mbox{if }j=1,\\
2p_{n-2}^{j-1}+2p_{n-2}^{j+1}, & \mbox{otherwise.}
\end{cases}
\end{equation}
The boxed term is the only one from the $(n-3)$th step and others are from the $(n-2)$th step. However, we cannot reduce it to the ones from the $(n-2)$th step. What we are going to do is to take only even-numbered grid points for the domain $x<0$ and even-numbered time steps $n=2\ell$. Then, \eqref{2.4} is written as
\begin{equation} \label{RW2}
    p^{j}_{2\ell} =\frac{1}{4}\times\begin{cases}
        p^{j-2} _{2(\ell-1)}+2p ^{j} _{2(\ell-1)}+p^{j+2}_{2(\ell-1)}, & \text{if $j<0$ is even,}\\
        p^{-2} _{2(\ell-1)}+p^{0} _{2(\ell-1)}+2p ^{1} _{2(\ell-1)}, & \text{if $j=0$,}\\
        2p^{j-1} _{2(\ell-1)} + 2p ^{j+1} _{2(\ell-1)}, & \text{if } j>0.
        \end{cases}
\end{equation}
The advantage of this formula is that the odd-numbered time indexes do not appear at all, so it is a Markov chain. Furthermore, the odd-numbered space indexes are completely forgotten in the region $x<0$. In other words, the even-numbered time indexes, the even-numbered space indexes in the region $x<0$, and all the space indexes in $x\ge0$ form an independent system. This is a special property when the sojourn time in one region is exactly twice as long as in the other. Below we will construct a discrete random walk system using these grid points.

Note that $p_n^j$ is a probability, not density, which is not appropriate to take the diffusion limit as $\eps\to0$ since it converges to 0 as $\eps\to0$. We should consider the \emph{probability density} which will be denoted by $v_n^j$ (or by $v(t,x)$ for the limit). We assume that the initial probability density distribution $v_0$ satisfies
\begin{equation} \label{initial}
\begin{cases}
 v_0 \in L^1( \R)\cap L^\infty(\R),\quad v_0\ge0,\\
 \int_ \R v_0(x)\, dx = 1, \\
 w_0:=\frac{v_0}{\tau} \in C^2(\R).
 \end{cases}
\end{equation}
Since a classical solution of the problem \eqref{ux1} requires the regularity of $\frac{v(x)}{\tau(x)}$, the smoothness assumption of the initial value should be given to $\frac{v_0(x)}{ \tau(x)}$, not to $v_0(x)$. Based on the relation \eqref{RW2}, we introduce a non-uniform space grids $(y^j)_{j\in\Z}$ given by
%\begin{equation}\label{yj}
\[
y^j:=x^{2j}\quad\text{if}\quad j<0,\quad\text{and}\quad
y^j:=x^{j}\quad\text{if}\quad j\ge0.
\]
%\end{equation}
Then, the mesh sizes are
\[
\Delta y^j:=y^{j+1}-y^j=\begin{cases}
                          2\eps , & \mbox{if } j<0,\\
                          \eps, & \mbox{if } j\ge0.
                        \end{cases}
\]
The initial probability density function $v_0(x)$ is discretized as
\begin{equation*}\label{p0k}
\tilde p_{0}^{j} := \int_{y^j}^{y^{j+1}} v_0(x)\,dx=\begin{cases}
                          \int_{y^j}^{y^{j}+2\eps } v_0(x)\,dx, & \mbox{if } j<0,\\
                          \int_{y^j}^{y^{j}+\eps } v_0(x)\,dx, & \mbox{if } j\ge0.
                        \end{cases}
\end{equation*}
Then, $\sum_{j}\tilde p_{0}^{j}=1$. Let $\tilde p_{\ell}^j := p_{2\ell}^{2j}$ if $j<0$ and $\tilde p_{\ell}^j := p_{2\ell}^{j}$ if $j\ge0$. Then, \eqref{RW2} is written as
\begin{equation} \label{RW3}
    \tilde p^{j}_{\ell} =\frac{1}{4}\times\begin{cases}
        \tilde p^{j-1} _{\ell-1}+2\tilde p ^{j} _{\ell-1}+\tilde p^{j+1}_{\ell-1}, & \text{if $j<0$,}\\
        \tilde p^{-1} _{\ell-1}+\tilde p^{0} _{\ell-1}+2\tilde p ^{1} _{\ell-1}, & \text{if $j=0$,}\\
        2\tilde p^{j-1} _{\ell-1} + 2\tilde p ^{j+1} _{\ell-1},
        & \text{if } j>0.
        \end{cases}
\end{equation}
The total mass is preserved and $\sum_{j}\tilde p_{\ell}^{j}=1$ for all $\ell\ge 0$.

Denote the probability density by $v_\ell^j$ and the initial ones by
\[
v_{0}^{j} ={1\over y^{j+1}-y^{j}}\int_{y^j}^{y^{j+1}} v_0(x)\,dx
=\begin{cases}
   \frac{1}{2\eps }\tilde p_0^j, & \mbox{if } j<0,\\
   \frac{1}{\eps}\,\tilde p_0^j, & \mbox{if } j\ge0.
 \end{cases}
\]
Since $v_{0}^{j}$ is not a scalar multiple of $\tilde p_{0}^{j}$, we cannot use the recursion relation \eqref{RW3} to construct $v_\ell^j$. Instead, we discretize the ratio $w(t,x)=\frac{v(t,x)}{\tau(x)}$. Denote $w_0:=\frac{v_0}{\tau}$ and discretize it in the interval $y^j<x<y^{j+1}$ by
\begin{equation*} \label{vol_aver+pmf}
w_{0}^{j} :={1\over y^{j+1}-y^{j}}\int_{y^j}^{y^{j+1}} \frac{v_0(x)}{\tau(x)}dx =\frac{1}{2\eps }\int_{y^j}^{y^{j+1}} v_0(x)dx=\frac{1}{2\eps }\tilde p_0^j,
\end{equation*}
where the heterogeneities in the mesh size $y^{j+1}-y^j$ and the sojourn time $\tau(x)$ work together to make $w_{0}^{j}$ a scalar multiple of $\tilde p_{0}^{j}$: $w_0^j=\frac{\tilde p_0^j }{2\eps }$ for all $j\in \Z$. Therefore, if we set
\[
w^{j}_\ell:={\tilde p_{\ell}^{j}\over2\eps }, \qquad \ell\in\Z_+,\ j\in\Z,
\]
then $w^{j}_\ell$ satisfy the recursion relation \eqref{RW3}, i.e.,
\begin{equation} \label{RW4}
\begin{aligned}
    &w^{j}_\ell = \left\{\begin{array}{ll}
        \frac{1}{4}w^{j-1} _{\ell-1} + \frac{1}{2}w ^{j} _{\ell-1} +
        \frac{1}{4} w^{j+1}_{\ell-1}, & \text{if $j<0$},\\
        \frac{1}{4}w^{-1} _{\ell-1} + \frac{1}{4}w^{0} _{\ell-1} +
        \frac{1}{2}w ^{1} _{\ell-1}, & \text{if $j=0$},\\
        \frac{1}{2}w^{j-1} _{\ell-1} + \frac{1}{2}w ^{j+1} _{\ell-1}, &
        \text{if $j>0$}.
        \end{array}\right.
\end{aligned}
\end{equation}

In the rest of this paper, we will show the existence of the diffusion limit $w(t,x)$ of the discrete model $w_\ell^j$ as $\eps\to0$. Then, the probability density distribution $v(t,x)$ of the limit $v_\ell^j$ is given by the relation,
$$
v(t,x)=\tau(x)w(t,x),
$$
where the initial value becomes $v(0,x)=v_0(x)$ as given in \eqref{initial}.

\section{Difference quotients}\label{sect.3}

Difference quotients are useful in finding the regularity of a weak solution and the differential equation satisfied by the limit of finite difference schemes. Let
\begin{equation*}\label{Ljl}
L^j_\ell := \frac{w_\ell^{j+1}-w_\ell^j}{y^{j+1}-y^j}
        =\left\{\begin{array}{ll}
        \frac{w^{j+1}_{\ell} - w^{j}_{\ell}}{2\eps } & \mbox{if } j<0,\\
        \frac{w^{j+1}_{\ell} - w^{j}_{\ell}}{\eps} & \mbox{if } j\ge0,
        \end{array}\right.
\end{equation*}
which is an approximation of the gradient $\partial_x w$. Let
\begin{equation}\label{Qjl}
Q^j_\ell := \frac{L^j_\ell-L^{j-1}_\ell}{2\eps }= \begin{cases}
                                                \frac{1}{\eps ^2}\big( \frac{1}{4} w^{j-1} _{\ell} -\frac{1}{2} w ^{j} _{\ell} +
        \frac{1}{4} w^{j+1} _{\ell}\big)& \mbox{if } j<0, \\
                                                \frac{1}{\eps ^2}\big( \frac{1}{4} w^{-1} _{\ell} - \frac{3}{4}w
        ^{0} _{\ell} +\frac{1}{2}w^{1} _{\ell}\big)& \mbox{if } j=0, \\
                                               \frac{1}{\eps ^2}\big(\frac{1}{2}w^{j-1} _{\ell} - w^j_\ell + \frac{1}{2}w ^{j+1}
        _{\ell}\big)& \mbox{if } j>0.
                                              \end{cases}
\end{equation}
Note that $y^{j+1}-y^{j}=2\eps $ for $j<0$ and $y^{j+1}-y^{j}=\eps$ for $j\ge0$. Therefore, $Q^j_\ell$ is an approximation of ${\partial^2_xw\over \tau}$, i.e., $Q^j_\ell$ approximates $\partial^2_xw$ for $j<0$ and  ${1\over2}\partial^2_xw$ for $j>0$. Note that it is not an approximation of $\partial_x({\partial_xw\over\tau})$. From \eqref{RW4}, we obtain
\begin{equation} \label{wjl}
    w^{j}_\ell-w^{j}_{\ell-1} = \begin{cases}
       \frac{1}{4}w^{j-1} _{\ell-1} - \frac{1}{2}w ^{j} _{\ell-1} +
        \frac{1}{4} w^{j+1}_{\ell-1} & \mbox{if } j<0, \\
        \frac{1}{4}w^{-1} _{\ell-1} - \frac{3}{4}w^{0} _{\ell-1} +
        \frac{1}{2}w ^{1} _{\ell-1} & \mbox{if } j=0, \\
        \frac{1}{2}w^{j-1} _{\ell-1}-w^{j}_{\ell-1} + \frac{1}{2}w ^{j+1}
        _{\ell-1}& \mbox{if } j>0.
                                              \end{cases}
\end{equation}
Therefore, by comparing \eqref{Qjl} and \eqref{wjl}, we obtain
\[
 \frac{w^j_\ell - w^j_{\ell-1}}{\eps^2} = Q^j_{\ell-1} = \frac{1}{2\eps }\Big( \Big(\frac{w^{j+1}_{\ell-1} -
 w^{j}_{\ell-1}}{y^{j+1}-y^j}\Big) -\Big(\frac{w^{j}_{\ell-1} -
 w^{j-1}_{\ell-1}}{y^{j}-y^{j-1}} \Big)\Big).
\]
In terms of $w_\ell^j$ only, we obtain
\begin{equation} \label{RW1w}
 \frac{w^j_\ell - w^j_{\ell-1}}{\eps^2} = \frac{1}{2\eps }\Big( \Big(\frac{w^{j+1}_{\ell-1} - w^{j}_{\ell-1}}{y^{j+1}-y^j}\Big) -\Big(\frac{w^{j}_{\ell-1} - w^{j-1}_{\ell-1}}{y^{j}-y^{j-1}} \Big)\Big).
\end{equation}
This is the main relation which will give the diffusion equation satisfied by the probability density function.

\subsection{Uniform estimates and continuous interpolation}
The relations in \eqref{RW4} can be used to find relations for $L^j_\ell$ and $Q^j_\ell$, which are
\begin{equation} \label{eq:L}
L^j_\ell = \begin{cases}
             \frac{1}{4}L_{\ell-1}^{j-1} + \frac{1}{2}L_{\ell-1}^j +
        \frac{1}{4}L_{\ell-1}^{j+1} & \mbox{if } j<0, \\
             \frac{1}{2}L_{\ell-1}^{j-1} + \frac{1}{2}L_{\ell-1}^{j+1} & \mbox{if }j\ge0,
           \end{cases}
\end{equation}
and
\begin{equation}\label{eq:Q}
Q^j_\ell =\begin{cases}
            \frac{1}{4}Q_{\ell-1}^{j-1} + \frac{1}{2}Q_{\ell-1}^{j} +
        \frac{1}{4}Q_{\ell-1}^{j+1} & \mbox{if } j<0, \\
            \frac{1}{4}Q_{\ell-1}^{-1} + \frac{1}{4}Q_{\ell-1}^{0} +
        \frac{1}{2}Q_{\ell-1}^{1} & \mbox{if } j=0, \\
            \frac{1}{2}Q_{\ell-1}^{j-1} + \frac{1}{2}Q_{\ell-1}^{j+1} & \mbox{if }j>0.
          \end{cases}
\end{equation}
Note that $L^j_\ell$ and $Q^j_\ell$ are weighted averages of the elements of the previous step. This observation allows for the necessary key estimates.

\begin{lemma}[Uniform estimates] \label{uniform_est}\
\begin{enumerate}\item The total sum of $w_\ell^j$ is preserved,
$$
\sum_{j} w^j_\ell = \sum_{j} w^j_0.
$$
\item Sequences $w_\ell^j$, $L_\ell^j$, and $Q_\ell^j$ are bounded by the
    initial maximums, i.e.,
  $$
  |w^j_\ell| \le \sup _{j} |w^j_0|,\quad  |L^j_\ell| \le \sup _{j}
  |L^j_0|,\quad |Q^j_\ell| \le \sup _{j} |Q^j_0|.
  $$
% \item We also have, for $p\ge1$,
% $$
% \sum _{j} |w^j_\ell|^p \le \sum _{j} |w^j_0|^p,\quad\sum _{j} |L^j_\ell|^p \le \sum _{j} |L^j_0|^p,\quad \sum _{j} |Q^j_\ell|^p \le
% \sum _{j} |Q^j_0|^p.
% $$
\end{enumerate}
\end{lemma}
\begin{proof} By \eqref{RW4},
\begin{align*}
 \sum_{j>0} w^j_\ell &= \sum_{j>0} \left(\frac{1}{2} w^{j-1}_{\ell-1} + \frac{1}{2}w^{j+1}_{\ell-1} \right) = \frac{1}{2} w^0_{\ell-1} - \frac{1}{2}w^1_{\ell-1} + \sum_{j>0} w^{j}_{\ell-1},\\
 w^0_\ell & = \frac{1}{4}w^{-1}_{\ell-1} + \frac{1}{4}w^0_{\ell-1} + \frac{1}{2} w^1_{\ell-1},\\
 \sum_{j<0} w^j_\ell &=\sum_{j<0}  \left(\frac{1}{4}w^{j-1}_{\ell-1} + \frac{1}{2}w^j_{\ell-1}+ \frac{1}{4}w^{j+1}_{\ell-1}\right)= - \frac{1}{4}w^{-1}_{\ell-1} + \frac{1}{4}w^0_{\ell-1} + \sum_{j<0} w^j_{\ell-1}.
\end{align*}
Therefore, $\displaystyle \sum_{j}{w^j_\ell} = \sum_j w^j_{\ell-1}$ and the first assertion follows. The second assertion is based on the fact that the three relations, \eqref{RW4},\eqref{eq:L}, and \eqref{eq:Q}, are in a form of weighted averages of the previous step. Therefore, the maximum should be in the initial value and the uniform estimate should be given in terms of initial values.
\end{proof}

To take the diffusion limit of a discrete model, we need to project the discrete values $w_\ell^j$ to a fixed continuum time-space $\R^+\times\R$. Then, we show the convergence of the continuum projections as $\eps\to0$. We construct it using four levels of interpolations. For $(t,y)\in D_\ell^j:=[t_{2\ell}, t _{2\ell+2}) \times [y^j, y^{j+1})$, denote $\lambda_\ell(t):=\frac{t-t_{2\ell}}{t_{2\ell+2}-t_{2\ell}}$ and $\nu^j(y):=\frac{y-y^j}{y^{j+1}-y^j}$, which takes values between 0 and 1 in the domain $D_\ell^j$. The four interpolations are
\begin{equation} \label{linear0}
\begin{aligned}
 w_1^\eps(t,y)&:= w^j_\ell,\\
 w_2^\eps(t,y)&: = \lambda_\ell(t) w^j_{\ell+1} + (1-\lambda_\ell(t))w^j_\ell,\\
 w_3^\eps(t,y)&: = \nu^j(y) w^{j+1}_{\ell} + (1-\nu^j(y))w^j_\ell,\\
  w^\eps(t,y)&:= \lambda_\ell(t)\nu^j(y)
  w_{\ell+1}^{j+1} + \lambda_\ell(t)(1-\nu^j(y))
  w_{\ell+1}^{j}\\
 &\qquad + (1-\lambda_\ell(t))\nu^j(y)
 w_{\ell}^{j+1} + (1-\lambda_\ell(t))(1-\nu^j(y))
 w_{\ell}^{j}.
  \end{aligned}
\end{equation}
The four projections have increasing higher regularities. The first one $w_1^\eps$ is constant on $D_\ell^j$ and piecewise constant in the space. The second one $w_2^\eps$ is piecewise linear in $t$ variable, piecewise constant in $y$ variable, and Lipschitz continuous in $t$ variable. The third one $w_3^\eps$ is piecewise linear in $y$ variable, piecewise constant in $t$ variable, and Lipschitz continuous in $y$ variable. The fourth projection $w^\eps$ is piecewise second order and Lipschitz continuous on the whole space $\R^+\times\R$. All of them take the value $w^j_\ell$ when $t=t_{2\ell}$ and $y=y^j$.

\begin{lemma} \label{misc}
Let $w_0\in W^{2,\infty}(\R)$ and $w^\eps$ be defined by $\eqref{linear0}$. Then, for any $\eps>0$ small and any $t>0$,
$$\|\partial_y w^\eps(t,\cdot)\|_{L^\infty(\R^+\times \R)} + \|\partial_t w^\eps(t,\cdot)\|_{L^\infty(\R^+\times \R)} \le C,$$
where $C$ depends only on $\|w_0\|_{W^{2,\infty}(\R)}$.
\end{lemma}
\begin{proof}
It is obvious that
$$|L_0^j| + |Q_0^j| \le C\|w_0\|_{W^{2,\infty}(\R)}.$$
By Lemma \ref{uniform_est}, \eqref{RW1w} and \eqref{linear0},
 \begin{align*}
 |\partial_t w^\eps(t,y)| \le \sup_j |Q^j_{\ell}|\le \sup_j |Q^j_{0}|,\quad \quad
 |\partial_y w^\eps(t,y)| \le \sup_j |L_0^j|.
 \end{align*}
 The proof is complete.
\end{proof}

We denote the space of Lipschitz continuous functions on $\R^+\times \R$ by $C^{0,1}\big(\R^+\times \R\big)$ below.

\begin{proposition}\label{compactness} There exists a subsequence $\{w^{\eps_i}\}$ and a function $w\in C^{0,1}\big(\R^+\times \R\big)$ such that $w^{\eps_i}\rightarrow w$ uniformly on any compact set $K \subset \R^+\times \R$ and $\partial_y w^{\eps_i} \rightharpoonup \partial_y w$ weakly$^*$ in $L^\infty(\R^+\times \R)$.
\end{proposition}
\begin{proof}
 This is due to the uniform bound in Lemma \ref{misc} and the compactness in $C^{0,1}(\R^+ \times \R)$.
\end{proof}

\begin{corollary} \label{cor_compactness}
 $w_1^{\eps_i},w_2^{\eps_i},w_3^{\eps_i}$ converges to the same limit $w$ uniformly on any compact set $K \subset \R^+\times \R$ and $\partial_y w_3^{\eps_i} \rightharpoonup \partial_y w$ weakly$^*$ in $L^\infty(\R^+\times \R)$.
\end{corollary}
\begin{proof} Proofs of the three cases are similar. We prove the assertion that $\partial_y w_3^{\eps_i}$ weakly$^*$ converges to the same limit $\partial_y w^{\eps_i}$ converges only. For a test function $g\in L^1(\R^+\times \R)$,
 \begin{align*}
  &\int_0^\infty\int_\R \Big|\partial_y {w}^{\eps_i} - \partial_y w_3^{\eps_i}\Big| \Big|g(t,y)\Big| \, dydt  \\
  &= \sum_{\ell\ge0}\int_{t_{2\ell}}^{t_{2(\ell+1)}}\int_\R \Big|\lambda_\ell(t)\Big| \: \Big|\partial_y w_3^{\eps_i}(t+2\eps_i^2,y)-\partial_y w_3^{\eps_i}(t,y)\Big| ~\Big|g(t,y)\Big| \,dydt\\
  &\le \int_0^\infty \int_\R \Big|\partial_y w_3^{\eps_i}(t+2\eps_i^2,y)-\partial_y w_3^{\eps_i}(t,y) \Big| \Big|g(t,y)\Big| \, dydt.
 \end{align*}
Since the integrand converges to $0$ pointwise, the difference converges to $0$ by the Lebesgue dominated convergence theorem.
\end{proof}

\section{Convergence to weak solution}\label{sect.4}

Now we show that the subsequential limit is a weak solution to a Cauchy problem of a non-autonomous parabolic equation,
\begin{equation}\label{eqnw}
   \tau(y) w_t = \frac{1}{2} w_{yy},\qquad w(0,y)=w_0(y).
\end{equation}

We define a weak solution of the limiting equation \eqref{eqnw} and show that the subsequential limit is a weak solution.

\begin{definition}[Weak solution] We call $w$ a \emph{weak solution} of
\eqref{eqnw} if $w, w_y \in L^1_{loc}(\R^+\times \R)$ and, for any test function $\phi \in C_c^1([0,\infty)\times \R)$,
 \begin{equation}\label{DefWeak}
  \int_0^\infty\int_\R  \tau w\phi_t\,dydt
  +\int_\R \tau w_0\phi(0,y)\,dy = \frac{1}{2}\int_0^\infty\int_\R w_y\phi_y\, dydt.
 \end{equation}
\end{definition}
\begin{theorem}\label{thm2.1}
Let $w_0\in C^{\infty}_c(\R)$ and $w^\eps$ be defined by $\eqref{linear0}$. Then, there exists a subsequence $\{w^{\eps_i}\}$ and a function $w\in C^{0,1}\big(\R^+\times \R\big)$ such that $w^{\eps_i}\rightarrow w$ uniformly on any compact set $K \subset \R^+\times \R$ and $\partial_y w^{\eps_i} \rightharpoonup \partial_y w$ weakly$^*$ in $L^\infty(\R^+\times \R)$. The subsequential limit is a weak solution of \eqref{eqnw}.
\end{theorem}
\begin{proof}
It remains to show that the subsequential limit $w^{\eps_i}\to w$ is a weak solution of \eqref{eqnw}. %Then, the uniqueness of energy decreasing solution gives the rest. The regularity of the limit $w$ is given in Lemma \ref{uniform_est}.
Consider the two terms on the left-hand side in \eqref{DefWeak}.
Since $w^{\eps_i}, w_1^{\eps_i}, w_2^{\eps_i} \rightarrow w$ uniformly on each compact set and $w_1^{\eps_i}(0,\cdot)\rightarrow w_0(\cdot)$ in $L^1$, the left-hand side converges as
\begin{align*}
&\int_0^\infty\int_\R w(t,y)  \tau(y)\partial_t\phi(t,y)\, dy
dt+\int_\R w_0(y) \tau(y)\phi(t,y) \, dy\\
=\:& \lim_{\eps_i\to0}\int_0^\infty\int_\R w^{\eps_i}(t,y)  \tau(y)\partial_t\phi(t,y)\, dy dt
 +\int_\R w^{\eps_i}(0,y) \tau(y)\phi(t,y) \, dy.
\intertext{In particular, this coincides with}
\quad &\lim_{\eps_i\to0}\int_0^\infty\int_\R w_2^{\eps_i}(t,y)  \tau(y)\partial_t\phi(t,y)\, dy dt
 +\int_\R w_2^{\eps_i}(0,y) \tau(y)\phi(t,y) \, dy.
\end{align*}
Denote a characteristic function by $\chi_j:=\chi_{[y^j,y^{j+1})}$. Then,
\begin{align*}
&\int_0^\infty\int_\R w_2^{\eps_i}(t,y)  \tau(y)\partial_t\phi(t,y)\, dy dt
 +\int_\R w_2^{\eps_i}(0,y) \tau(y)\phi(t,y) \, dy\\
 &\quad=-\sum_{\ell}\int_{t_{2\ell}}^{t_{2\ell+2}}\int_\R \partial_t
 w_2^{\eps_i}(t,y)  \tau(y)\phi(t,y)\, dy dt \\
 &\quad=-\sum_{\ell}\sum_{j}\int_{t_{2\ell}}^{t_{2\ell+2}}\int_\R
 \frac{w^j_{\ell+1}-w^j_{\ell}}{2\eps_i^2}  \tau\phi \chi_j(y)\, dy dt\\
 &\quad=-\sum_{\ell}\sum_{j}\frac{1}{2}\int_{t_{2\ell}}^{t_{2\ell+2}}\int_\R \frac{1}{2\eps_i } \Big( \Big(\frac{w^{j+1}_{\ell} -
 w^{j}_{\ell}}{y^{j+1}-y^j}\Big) -\Big(\frac{w^{j}_{\ell} -
 w^{j-1}_{\ell}}{y^{j}-y^{j-1}} \Big)\Big)  \tau\phi \chi_j(y)\, dy
 dt\\
 &\quad=-\sum_{\ell}\sum_{j}\frac{1}{2}\int_{t_{2\ell}}^{t_{2\ell+2}}\int_\R
 \frac{1}{2\eps_i } \Big(\frac{w^{j}_{\ell} -
 w^{j-1}_{\ell}}{y^{j-1}-y^j}\Big)
 \big(\chi_{j-1}(y)-\chi_{j}(y)\big) \tau\phi \, dy dt\\
 &\quad=-\sum_{\ell}\sum_{j}\frac{1}{2}\int_{t_{2\ell}}^{t_{2\ell+2}}B_\ell^j\,dt,
\end{align*}
where the third equality is from \eqref{RW1w} and the last term $B^j_\ell$ denotes
\[
B_\ell^j:=\int_\R \frac{1}{2\eps_i } \Big(\frac{w^{j}_{\ell} -
w^{j-1}_{\ell}}{y^{j-1}-y^j}\Big)
\big(\chi_{j-1}(y)-\chi_{j}(y)\big) \tau\phi \, dy.
\]
Then,
\begin{align*}
 B_\ell^j &= \Big(\frac{w^{j}_{\ell} - w^{j-1}_{\ell}}{y^{j}-y^{j-1}}\Big)
        \left\{\frac{1}{2\eps_i }\int_{y^{j-1}}^{y^{j}}  \tau\phi  \, dy -
        \frac{1}{2\eps_i }\int_{y^{j}}^{y^{j+1}}  \tau\phi  \, dy\right\}\\
      &=\Big(\frac{w^{j}_{\ell} - w^{j-1}_{\ell}}{y^{j}-y^{j-1}}\Big)
        \bigg( \frac{1}{y^j-y^{j-1}}\int_{y^{j-1}}^{y^{j}} \phi  \, dy-
        \frac{1}{y^{j+1}-y^{j}}\int_{y^{j}}^{y^{j+1}} \phi  \, dy
        \bigg)\\
      &= -\int_{y^{j-1}}^{y^{j}}\Big(\frac{w^{j}_{\ell} -
 w^{j-1}_{\ell}}{y^{j}-y^{j-1}}\Big)\partial_y\phi \,dy + \mathcal{E}^j_\ell,
\end{align*}
where
\[
\mathcal{E}^{j}_\ell = \Big(\frac{w^{j}_{\ell} - w^{j-1}_{\ell}}{y^{j}-y^{j-1}}\Big) \Big(\phi(y^j)-\phi(y^{j-1}) + \frac{1}{y^j-y^{j-1}}\int_{y^{j-1}}^{y^{j}} \phi  \, dy-
        \frac{1}{y^{j+1}-y^{j}}\int_{y^{j}}^{y^{j+1}} \phi  \, dy
        \Big).
\]
By writing
\begin{align*}
 \frac{\int_{y^{j}}^{y^{j+1}} \phi  \; dy}{y^{j+1}-y^{j}} - \phi(y^{j}) &=\big(y^{j+1}-y^j\big) \int_0^1\int_0^\lambda \partial_y\phi \big(sy^{j+1} + (1-s)y^j\big) \,ds d\lambda\\
 \frac{\int_{y^{j-1}}^{y^{j}} \phi  \; dy}{y^{j}-y^{j-1}} - \phi(y^{j-1}) &= \big(y^{j}-y^{j-1}\big) \int_0^1\int_0^\lambda  \partial_y\phi \big(sy^{j} + (1-s)y^{j-1}\big)\,ds d\lambda,
\end{align*}
% \begin{align*}
% |\mathcal{E}^{j}_\ell| = \bigg|\frac{w^{j}_{\ell} - w^{j-1}_{\ell}}{y^{j}-y^{j-1}}\bigg|
%  \:\bigg|\int_0^1\int_0^\lambda &\phi^{\prime} \big(sy^{j+1} + (1-s)y^j\big)\big(y^{j+1}-y^j\big) \\
% -&\phi^{\prime} \big(sy^{j} + (1-s)y^{j-1}\big)\big(y^{j}-y^{j-1}\big)\:dsd\lambda\bigg|.
% \end{align*}
we observe $|\mathcal{E}^j_\ell| = o(\eps)$ if $j\ne0$ and $o(1)$ if $j=0$ as $\epsilon \rightarrow 0$. Since $\phi$ is compactly supported, $\sum_j |\mathcal{E}^j_\ell| = o(1)$ as $\epsilon \rightarrow 0$ and $\sum_\ell\sum_j |\mathcal{E}^j_\ell| = o(1)$ as $\epsilon \rightarrow 0$.
Therefore, using Corollary \ref{cor_compactness} we obtain 
\[
-\sum_{\ell}\sum_{j}\frac{1}{2}\int_{t_{2\ell}}^{t_{2\ell+2}}B_\ell^j\,dt = \frac{1}{2}\int_{0}^{\infty}\int_\R \partial_y w_3^{\eps_i}(t,y) \partial_y \phi(t,y) \,
dy + o(1)
\]
which converges to
\[
\frac{1}{2}\int_0^\infty\int_\R \partial_y w(t,y) \partial_y \phi(t,y) \, dy.
\]
Therefore, $w$ is a weak solution.
\end{proof}

We now show that there is an energy function such that the energy level of a weak solution decreases as $t\to\infty$. Let $w$ be a smooth solution of \eqref{eqnw} and define an energy function by
\[
e(t;w):={1\over 2}\int  \tau(y) w^2(t,y)\,dy.
\]
If $w(t,\cdot)$ vanishes at infinity for all $t > 0$, the energy function decreases in time:
\[
\frac{d}{dt}e(t;w)=\int  \tau ww_t\,dy=\frac{1}{2}\int ww_{yy}\,dy=-\frac{1}{2}\int w_y^2\,dy\le 0.
\]
We show that the energy of a weak solution also decreases in time if it is a subsequential limit found in Theorem \ref{thm2.1}.

\begin{theorem} \label{thmL1L2}
Let $w$ be a weak solution of \eqref{eqnw} and $w^{\eps_i}$ is the subsequence in Theorem \ref{thm2.1} that converges to $w$. Then,
\begin{enumerate}
 \item For all $t\in[0,\infty)$, $\int_\R \tau(y) w(t,y)\,dy = 1$.
 \item For all $t>0$, $e(t;w)\le e(0;w)$.
\end{enumerate}
\end{theorem}
\begin{proof}
For $t\in [t_{2\ell}, t_{2\ell+2})$, we have
\begin{align*}
\int_\R  \tau(y)w_1^{\eps_i}(t,y) \, dy &= \sum_{j<0}w^j_\ell2{\eps_i}  + \sum_{j\ge0}2w^j_\ell \eps_i\\
&=\sum_{j<0}p^{2j}_{2\ell}+ \sum_{j\ge0}p^j_{2\ell}=
\sum_{j<0}p^{2j}_{0}+ \sum_{j\ge0}p^j_{0}  = 1.
\end{align*}
Therefore, since $w_1^{\eps_i}(t,\cdot) \rightarrow w(t,\cdot)$ in $L^1(\R)$,
$\int_\R \tau(y) w(t,y) \, dy = 1$.

Next consider the energy of the approximation $w^{\eps_i}_1$,
\begin{align*}
e(t_{2\ell};w^{\eps_i}_1)&=\int_\R  \tau(y) (w^{\eps_i}_1(t_{2\ell},y))^2 \,dy\\
&=\sum_{j<0}(w_\ell^{j})^2(2{\eps_i} ) +\sum_{j\ge0}2(w_\ell^{j})^2{\eps_i}  =\sum_{j<0}(p_{2\ell}^{2j})^2 +\sum_{j\ge0}(p_{2\ell}^{j})^2 .
\end{align*}
Since $f(p)=p^2$ is a convex function, \eqref{RW2} gives that
$$
\begin{aligned}
    &(p^{j}_{2\ell})^2 \le \left\{\begin{array}{ll}
        \frac{1}{4}(p^{j-2} _{2(\ell-1)})^2 + \frac{1}{2}(p^{j}_{2(\ell-1)})^2 +\frac{1}{4} (p^{j+2}_{2(\ell-1)})^2, & \text{if $j<0$ is even},\\
        \frac{1}{4}(p^{-2} _{2(\ell-1)})^2 + \frac{1}{4}(p^{0} _{2(\ell-1)})^2 +\frac{1}{2}(p ^{1} _{2(\ell-1)})^2, & \text{if $j=0$},\\
        \frac{1}{2}(p^{j-1} _{2(\ell-1)})^2 + \frac{1}{2}(p ^{j+1} _{2(\ell-1)})^2,
        & \text{if } j>0.
        \end{array}\right.
\end{aligned}
$$
Therefore, $e(t_{2\ell};w^{\eps_i}_1)\le e(0;w^{\eps_i}_1)$ for all $\ell>0$. Since $w_1^{\eps_i}(t,\cdot) \rightarrow w(t,\cdot)$ in $L^1(\R)$ and $|w_1|,|w|$ are uniformly bounded, by taking limit $\eps_i \rightarrow 0$ we have $e(t;w)\le e(0;w)$ for all $t>0$.
\end{proof}

The probability density function is $v(t,x)= \tau(x)\, w(t,x)$. Then, 
\[
v_t(t,x)=\tau(x)w_t(t,x)=\frac{1}{2}w_{xx}(t,x)=\frac{1}{2}\Big(\frac{v(t,x)}{\tau(x)}\Big)_{xx}
\]
and hence $v(t,x)$ satisfies
\begin{equation}\label{eqnv}
v_t=\frac{1}{2}\Big({v\over \tau(x)}\Big)_{xx},\qquad
v(0,x)={ \tau(x)w_0(x)}=v_0(x),
\end{equation}
where the initial value $v_0(x)$ is the original initial value given in \eqref{initial}. Let $v$ be defined by $\tau w$ with $w$ obtained in Theorem \ref{thm2.1} and define an energy function by
\[
e(t;v)={1\over 2}\int  \frac{1}{\tau(x)} v^2(t,x)\,dx.
\]
The energy function decreases in time by Theorem \ref{thmL1L2}.
% \[
% \frac{d}{dt}e(t;v)=\int  \frac{v}{\tau} v_t\,dx=\frac{1}{2}\int \frac{v}{\tau}\Big(\frac{v}{\tau}\Big)_{xx}\,dx=-\frac{1}{2}\int\Big(\frac{v}{\tau}\Big)_x^2\,dx\le 0.
% \]

\section{Green's function} \label{sect.5}

In this section we find the Green's function $G=G(t,x;a)$ of the diffusion equation, which is an explicit solution of
\begin{equation}\label{GV}
\begin{cases}
v_t=(\frac{1}{\tau}v)_{xx},\\
v(0,x)=\delta(x-a),
\end{cases}
\end{equation}
where $a\in\R$ is a given number and $\delta(x)$ is the Dirac delta distribution. Note that we have omitted $\frac{1}{2}$ in the equation \eqref{eqnw} for convenience; this can be done by doubling the scale of the time variable. Denote
\begin{equation}\label{mu}
\gamma(y):=\frac{1}{\tau(y)}=\begin{cases}
1&\mbox{ if } y<0,\\
\frac{1}{2}&\mbox{ if } y>0.
\end{cases}
\end{equation}
In the following, we will look for the Green's function such that $\gamma(\cdot)G(t,\cdot;a)\in C^1(\R)\cap C^\infty(\R\setminus\{0\})$ for each $a\in\R$ and $t>0$. We will take the transformation $v=\tau w$ and first find the Green's function of the transformed problem for $w$, which was the same strategy used to find the diffusion limit in the previous sections. Then, $w$ satisfies
\begin{equation}\label{5.2}
\begin{cases}
w_t=\gamma(y)w_{yy},\\
w(0,y)=\gamma(y)\delta(y-a).
\end{cases}
\end{equation}
We denote $W(t,y;a)$ by the Green's function that satisfies \eqref{5.2}.

\subsection{Explicit formula for Green's function}
We divide the problem into two cases. First, we consider the case where $a=0$ and find an explicit function $w(t,y)$ that solves
\begin{equation}\label{w0}
\left\{
\begin{aligned}
&w_t= \gamma(y) \,w_{yy},\\
&w(0,y)=\gamma(y) \,\delta(y).
\end{aligned}
\right.
\end{equation}
This is the most difficult case since the discontinuity of the coefficient $\gamma$ is at the origin and the Dirac delta distribution is also placed there. We use a technique used in \cite{010}. To begin, we divide the real line $\R$ into two regions, $\{ y>0 \}$ and $\{ y<0 \}$. In the region $\{y>0\}$, $\gamma=\frac{1}{2}$ and we solve an initial-boundary value problem:
\begin{equation}\label{w+}
\begin{cases}
w_t = \frac{1}{2} w_{yy} & \text{if $y > 0, ~t > 0$,} \\
w(0,y) = \frac{1}{2} f(y) & \text{if $y > 0$,} \\
w(t,0) = g(t) & \text{if $t > 0$,}
\end{cases}
\end{equation}
where the interface value $g(t)$ at $y=0$ is unknown and will be determined later. Note that we have replaced the initial data $\delta$ by an arbitrary function $f(y)$ for now. This will be replaced by a Dirac delta sequence and the unknown boundary value $g$ will be decided from a limiting process. One may find the solution of the initial-boundary value problem \eqref{w+} from \cite[pp.18--19 and p.22]{kevorkian}, which is
\[ \begin{split}
w(t,y) &= \frac{1}{\sqrt{2 \pi t}} \int_0^\infty \frac{1}{2} f(\xi) \big( e^{-(\xi-y)^2/2t} - e^{-(\xi+y)^2/2t} \big)\,d\xi \\
& \qquad + \int_0^t g'(\eta) \,\erfc\Big( \frac{y}{\sqrt{2(t-\eta)}} \Big) \,d\eta + g(0) \,\erfc\Big( \frac{y}{\sqrt{2t}} \Big) \quad \text{for\  $y > 0$.}
\end{split} \]
In the region $\{y<0\}$, $\gamma=1$ and we solve another initial-boundary value problem:
\begin{equation}\label{w-}
\begin{cases}
w_t = w_{yy} & \text{if $y<0, ~t > 0$,} \\
w(0,y) = f(y) & \text{if $y < 0$,} \\
w(t,0) = g(t) & \text{if $t > 0$.}
\end{cases}
\end{equation}
Note that, by choosing the same boundary value $g(t)$ in \eqref{w-}, the solution $w$ is forced to be continuous at the interface $y=0$. One can similarly find its solution
\[ \begin{split}
w(t,y) &= \frac{-1}{2\sqrt{\pi t}} \int_0^\infty f(-\xi) \big( e^{-(\xi-y)^2/4t} - e^{-(\xi+y)^2/4t} \big) \,d\xi \\
& \qquad + \int_0^t g'(\eta) \,\erfc\Big( \frac{-y}{2\sqrt{t-\eta}} \Big) \,d\eta + g(0) \,\erfc\Big( \frac{-y}{2\sqrt{t}} \Big) \quad \text{for\  $y < 0$.}
\end{split} \]

To determine $g(t)$, we use the assumption that the solution $w$ is $C^1(\R)$ in the $y$ variable, i.e.,
\[ w_y(t,0+) = w_y(t,0-) \qquad \text{for all $t > 0$.} \]
Since
\[w_y(t,y)=
\begin{cases}
\begin{split}
& \textstyle \frac{1}{2t \sqrt{2\pi t}} \int_0^\infty f(\xi) \big( (\xi-y) e^{-(\xi-y)^2/2t} + (\xi+y) e^{-(\xi+y)^2/2t} \big) \,d\xi \\
& \textstyle \qquad - \sqrt{\frac{2}{\pi}} \int_0^t g'(\eta) \, \frac{e^{-y^2/2(t-\eta)}}{\sqrt{t-\eta}} \,d\eta - \sqrt{\frac{2}{\pi t}} \,g(0) \,e^{-y^2/2t},\ \text{if}\ y>0,
\end{split}
&\\
\begin{split}
& \textstyle \frac{-1}{4t \sqrt{\pi t}} \int_0^\infty f(-\xi) \big( (\xi-y) e^{-(\xi-y)^2/4t} + (\xi+y) e^{-(\xi+y)^2/4t} \big) \,d\xi \\
& \textstyle \qquad + \frac{1}{\sqrt{\pi}} \int_0^t g'(\eta) \,\frac{e^{-y^2/4(t-\eta)}}{\sqrt{t-\eta}} \,d\eta + \frac{1}{\sqrt{\pi t}} \,g(0) \,e^{-y^2/4t},\ \text{if}\ y<0.
\end{split}
&
\end{cases}\]
the continuous differentiability of $w$ implies that
\[ \begin{split}
& \frac{1}{t \sqrt{2t}} \int_0^\infty f(\xi) \,\xi e^{-\xi^2/2t}\,d\xi + \frac{1}{2t \sqrt{t}} \int_0^\infty f(-\xi) \, \xi e^{-\xi^2/4t} \,d\xi \\
& \qquad\qquad = (\sqrt{2}+1) \int_0^t g'(\eta) \, \frac{1}{\sqrt{t-\eta}} \,d\eta + (\sqrt{2} + 1) \,\frac{g(0)}{\sqrt{t}} \quad \text{for all $t > 0$.}
\end{split} \]
This relation gives the boundary value $g(t)$ implicitly for any given initial value $f(y)$. For the Green's function case, the corresponding initial value is the Dirac delta distribution and we can find an explicit formula for $g(t)$. To do so, we take a Dirac delta sequence as the initial value: we set $f = \frac{n}{2} \,\chi_{[-1/n,1/n]}$ for a given positive integer $n$. Then, the implicit relation is written as
\[ \begin{split}
&\frac{n}{2\sqrt{2t}} (1 - e^{-\frac{1}{2n^2 t}}) + \frac{n}{2\sqrt{t}} (1 - e^{-\frac{1}{4n^2 t}}) \\
& \qquad\qquad = (\sqrt{2} + 1) \int_0^t g'(\eta) \,\frac{1}{\sqrt{t-\eta}} \,d\eta + (\sqrt{2} + 1) \, \frac{g(0)}{\sqrt{t}}.
\end{split} \]
Taking the Laplace transformation yields
\[ \begin{split}
& \frac{n}{2\sqrt{2}} \sqrt{\frac{\pi}{s}} (1 - e^{-\sqrt{2s}/n}) + \frac{n}{2} \sqrt{\frac{\pi}{s}} (1 - e^{-\sqrt{s}/n} ) \\
& \hskip 2cm = (\sqrt{2} + 1) (s \mathcal{G} (s) - g(0)) \sqrt{\frac{\pi}{s}} + (\sqrt{2} + 1) \,g(0) \sqrt{\frac{\pi}{s}}\ ,
\end{split} \]
where $\mathcal{G} (s)$ is the Laplace transformation of $g$. Therefore
\[ 
\mathcal{G} (s) = \frac{n}{2(\sqrt{2}+1) s} \Big( \frac{1}{\sqrt{2}} (1 - e^{-\sqrt{2s}/n}) + (1 - e^{-\sqrt{s}/n}) \Big),
\]
and the inverse Laplace transform gives
\[ g(t) = \frac{n}{2} (\sqrt{2} - 1) \Big( \frac{1}{\sqrt{2}} \erf\Big( \frac{\sqrt{2}}{2n\sqrt{t}} \Big) + \erf \Big( \frac{1}{2n \sqrt{t}} \Big) \Big). \]
The boundary condition $g(t)$ is obtained by taking the limit as $n \to \infty$:
\[ g(t) = \frac{\sqrt{2} - 1}{\sqrt{\pi t}}. \]
Therefore, the solution $w$ of (\ref{w0}) is given by
\[w(t,y)=\begin{cases}
\frac{\sqrt{2}-1}{\sqrt{2} \pi} \int_0^t \frac{y e^{-y^2/2(t-\tau)}}{\sqrt{\tau} (t-\tau)^{3/2}} \,d\tau = \frac{2 (\sqrt{2}-1)}{\pi} \int_{y/\sqrt{2t}}^\infty \frac{e^{-\eta^2}}{\sqrt{t - y^2/2\eta^2}} \,d\eta & \text{if $y > 0$,} \\
\frac{\sqrt{2}-1}{2\pi} \int_0^t \frac{(-y) e^{-y^2/4(t-\tau)}}{\sqrt{\tau} (t-\tau)^{3/2}} \,d\tau = \frac{2 (\sqrt{2}-1)}{\pi} \int_{-y/2\sqrt{t}}^\infty \frac{e^{-\eta^2}}{\sqrt{t-y^2/4\eta^2}}\,d\eta & \text{if $y < 0$.}
\end{cases}\]
Two different cases above can be combined using $\gamma(y)$: the Green's function for \eqref{5.2} when $a=0$ is
\begin{equation}\label{W0}
W(t,y;a=0)=\frac{2(\sqrt{2}-1)}{\pi} \int_{|y|/\sqrt{4\gamma(y)t}}^\infty \frac{e^{-\eta^2}}{\sqrt{t-y^2/4\gamma(y)\eta^2}}\,d\eta.
\end{equation}

Next, we consider the other case where $a\ne0$. The Green's function can be constructed in the same manner by dividing the spatial domain into two regions,  $\{y>0\}$ and $\{y<0\}$. For the case when $a > 0$, the corresponding boundary value is
\[
g(t)= \frac{\sqrt{2} - 1}{\sqrt{\pi t}} \,e^{-a^2/2t}.
\]
Hence, the Green's function of \eqref{5.2} for $a>0$ is given by
\begin{equation}\label{W+}
W(t,y;a>0)=
\begin{cases}
\begin{split}
& \textstyle \frac{2(\sqrt{2}-1)}{\pi} \int_{|y|/\sqrt{2 t}}^\infty \exp\Big(- \frac{a^2/2}{t - y^2/2\eta^2}\Big)  \, \frac{e^{-\eta^2}}{\sqrt{t - y^2/2\eta^2}} \,d\eta\\
& \textstyle \qquad + \frac{1}{2\sqrt{2\pi t}} \big( e^{-(y-a)^2/2t} - e^{-(y+a)^2/2t} \big)
\end{split}
& \text{if\  $y>0$,} \\
\frac{2(\sqrt{2}-1)}{\pi} \int_{|y|/2\sqrt{t}}^\infty \exp\Big( -\frac{a^2/2}{t-y^2/4\eta^2} \Big) \, \frac{e^{-\eta^2}}{\sqrt{t - y^2/4\eta^2}} \,d\eta & \text{if\  $y < 0$.}
\end{cases}
\end{equation}
For the case where $a < 0$, we have
\[ g(t) = \frac{\sqrt{2} - 1}{\sqrt{\pi t}} \,e^{-a^2/4t} \]
and
\begin{equation}\label{W-}
W(t,y;a<0)=\begin{cases}
\frac{2(\sqrt{2}-1)}{\pi} \int_{|y|/\sqrt{2t}}^\infty \exp\Big( -\frac{a^2/4}{t-y^2/2\eta^2} \Big) \, \frac{e^{-\eta^2}}{\sqrt{t - y^2/2\eta^2}} \,d\eta & \text{if\  $y > 0$,} \\
\begin{split}
& \textstyle \frac{2(\sqrt{2}-1)}{\pi} \int_{|y|/2\sqrt{t}}^\infty \exp\Big( - \frac{a^2/4}{t - y^2/4\eta^2} \Big) \, \frac{e^{-\eta^2}}{\sqrt{t - y^2/4\eta^2}} \,d\eta \\
& \textstyle \quad + \frac{1}{2\sqrt{\pi t}} \big( e^{-(y-a)^2/4t} - e^{-(y+a)^2/4t} \big)
\end{split}
& \text{if\  $y < 0$.}
\end{cases}
\end{equation}
Note that the formulas for the two cases where $a>0$ and $a<0$ are almost identical; especially, the integrands for the case where $a < 0$ is achieved just by simply replacing `$a^2/2$' in the exponent for the case where $a > 0$ by `$a^2/4$'. 

The Green's function $G$ for the problem \eqref{GV} is obtained by multiplying $\frac{1}{\gamma}$ to the Green's function $W$ of \eqref{5.2}, i.e.,
\begin{equation}\label{G}
G(t,x;a)=\frac{1}{\gamma}W(t,x;a),
\end{equation}
where $W$ is given by \eqref{W0}, \eqref{W+}, and \eqref{W-}, and $\gamma$ is by \eqref{mu}.

\subsection{Numerical comparison with the random walk}

We compare the probability density function of the discrete random walk system \eqref{RW1} and the explicit Green's function in \eqref{G} to see how similar they are. The Green's function given in \eqref{G} is for the equation \eqref{GV}. If the equation is given by
\begin{equation}\label{GVD}
v_t=D \Big( \frac{1}{\tau}v \Big)_{xx}, \qquad v(0,x)=\delta(x-a),
\end{equation}
for a given constant $D>0$, the solution of \eqref{GV} at $t=t_0$ is the same as the solution of \eqref{GVD} at $t=t_0/D$.

\begin{figure}[ht]
\small
\centering
\begin{minipage}[t]{0.49\textwidth}
 \centering
 \includegraphics[width=\textwidth]{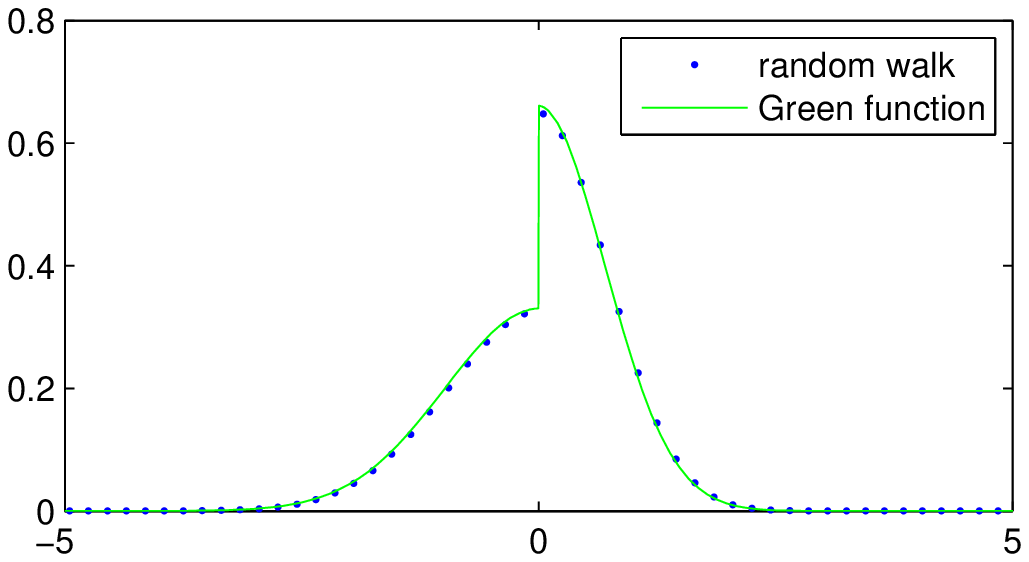}
(a) $\tri x=0.1$
\end{minipage}
\begin{minipage}[t]{0.49\textwidth}
\centering
 \includegraphics[width=\textwidth]{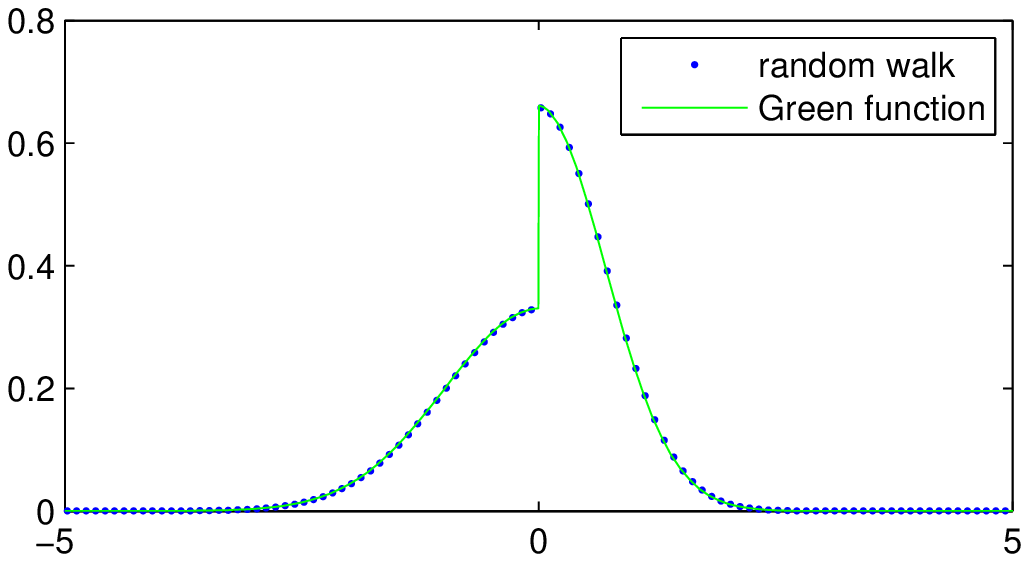}
(b) $\tri x=0.05$
\end{minipage}
\caption{\small{\bf Green's function and random walk.} Probability distributions of the random walk (dots) models and the explicit solution $G(t=0.5,x;a=0)$ in \eqref{G} (solid line).}\label{fig1}
\end{figure}

In Figure \ref{fig1}(a), the probability distribution of the discrete random walk system starting with the initial value $p_0^j=\delta_{0j}$ (Kronecker delta) is given when the space mesh width is $\tri x=0.1$ and the number of time steps is $n=100$. Here the diffusion limit as $\eps \to 0$ is considered with
\[
\tri t=\eps^2\tau,\quad \tri x=\eps\ell,
\]
where $\ell=0.1$ and $\tau$ is the one defined in \eqref{1.1}.
Hence, after taking the diffusion limit, we obtain the diffusion equation \eqref{GVD} with $D=\frac{\ell^2}{2}=0.005$. Hence, the $100$th time step  corresponds to the time $t_0=100\times0.005=0.5$. In Figure \ref{fig1}(a), Green's function $G(t=0.5,x;a=0)$ is given in a solid line. We can see that they match.

In Figure \ref{fig1}(b), the probability distribution of the discrete random walk system is given when the space mesh width is $\tri x=0.05$ and $n=400$. Similarly, after taking the diffusion limit, we obtain the diffusion equation \eqref{GVD} with $D=0.00125$. Hence, the 400th time step corresponds to the time $t_0=400\times0.00125=0.5$ again. In Figure \ref{fig1}(b), we can see that they match.

\section{Monte Carlo simulation}\label{sect.6}

In this section, we present three Monte Carlo simulation results that illustrate the properties of a random walk system with a nonconstant sojourn time. The computation codes for the three examples are given in the Appendix. Figure \ref{fig2} shows the result of a Monte Carlo simulation where the particle distribution corresponds to the distribution of the Green's function at $t=0.5$. The walk length $\tri x$ takes the normal distribution with a deviation of $0.1$. Then, the diffusivity in the region where $\tri t=1$ is
\begin{equation}\label{6.D}
D=\frac{|\tri x|^2}{2}=0.005.
\end{equation}
The time scale $t=0.5$ of the Green's function is the one when the diffusivity is $1$. Hence, if the diffusivity is $D=0.005$, the corresponding time is
\begin{equation}\label{6.T}
T=\frac{t}{D}=100.
\end{equation}
For the simulation, we let $10^6$ particles walk a maximum of 100 times each. 
If a particle walks in the region $x>0$, one walk is counted as two. So if a particle is always in the region $x>0$, it will walk 50 times. The distribution of the final positions of the $10^6$ particles is plotted in Figure \ref{fig2}. We can see an almost exact match with the graphs in Figure \ref{fig1}. In other words, the Green's function \eqref{G} and hence the diffusion equation \eqref{ux1} correctly provide the evolution of the random walk with a heterogeneous sojourn time given in \eqref{tau(x)}.

\begin{figure}[ht]
\small
\centering
 \includegraphics[width=0.7\textwidth]{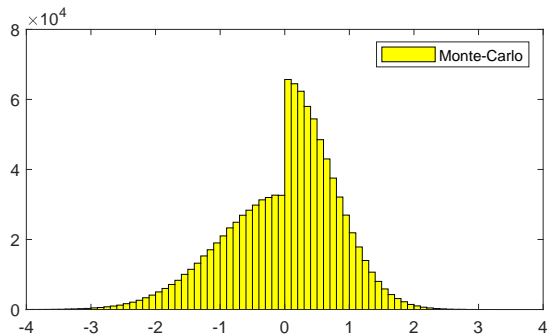}
\caption{\small{\bf Histogram of particle positions.} We let $10^6$ particles walk for $400$ steps starting from the origin with walk length $\tri x=0.05$. This histogram of the final particle positions shows the agreement with the Green's function.} \label{fig2}
\end{figure}

The explicit formula of the Green's function \eqref{G} allows us to compare Monte Carlo simulations and the exact solution when all particles start from a single position. Another way to validate the model equation \eqref{ux1} is to check that it has a valid steady-state solution. We can see that the steady state of the differential equation \eqref{ux1} is proportional to the sojourn time, i.e., there exists a constant $C>0$ such that
\[
u(x,t)\to C\tau(x)\quad\text{as}\quad t\to\infty.
\]
Since the initial total mass is preserved, the constant $C$ satisfies
\begin{equation}\label{6.C}
C=\frac{\int u_0(x)dx}{\int\tau(x)dx} .
\end{equation}
In the convergence proof of this paper, we have considered a case where the domain is divided into two parts $\{x<0\}$ and $\{x>0\}$ and the sojourn time is $\tri t=1$ and $\tri t=2$, respectively. However, the model equation itself is general and can be used for a general situation. For example, we may take a bounded domain with an appropriate boundary conditions or a general function $\tau(x)$ as the sojourn time. One of the simplest ways to test the convergence to a steady state is to take a bounded domain with a periodic boundary condition. Figure \ref{fig3} shows a Monte Carlo simulation is given with the sojourn time
\begin{equation}\label{6.1}
\tau(x)=\begin{cases}
1,&\text{ if } 0<x<1,\\
2,&\text{ if } 1<x<2,\\
3,&\text{ if } 2<x<3,\\
4,&\text{ if } 3<x<4,
\end{cases}
\end{equation}
on the interval $\Omega=(0,4)$ and with the periodic boundary condition. A total number of $10^6$ particles are used in the simulation. The initial distribution is taken with a random distribution. The walk length $\tri x$ is chosen from the normal distribution with a deviation of $0.1$. The particle distributions are given at four moments, $t=0, 0.2, 1$, and $5$. The actual number of walks is given as above after calculating the diffusivity. We can see that the histogram converges to a distribution proportional to the sojourn time $\tau(x)$.

\begin{figure}[ht]
\small
\centering
 \includegraphics[width=0.33\textwidth]{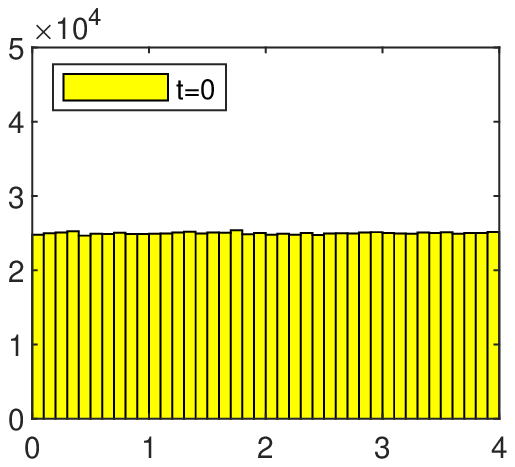}
 \includegraphics[width=0.33\textwidth]{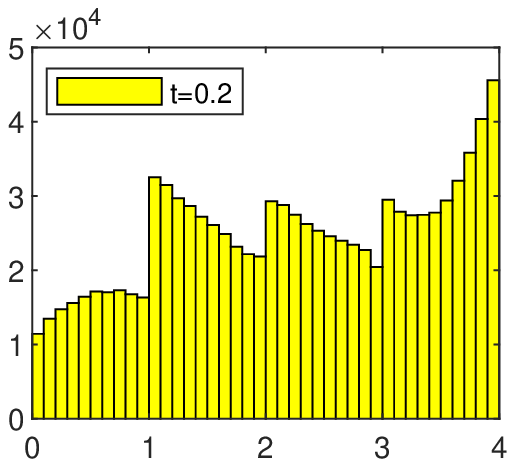}\\
 \includegraphics[width=0.33\textwidth]{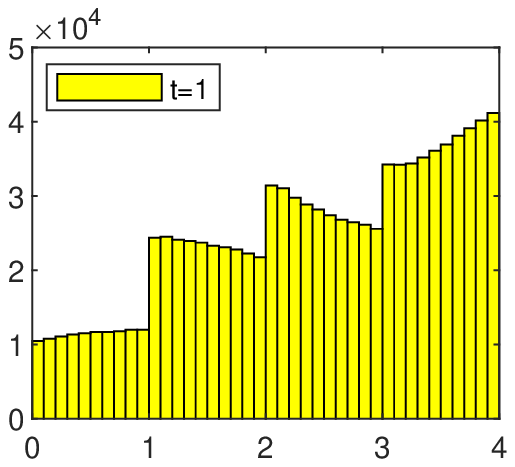}
 \includegraphics[width=0.33\textwidth]{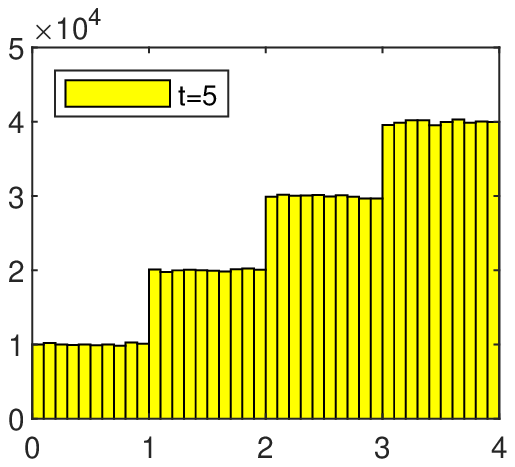}
\caption{\small{\bf Histogram and steady states.} Snapshots of the particle distribution when $\tri x=0.1$ and the sojourn time $\tri t=\tau$ is given by \eqref{6.1}. The steady state is proportional to $\tau(x)$. }\label{fig3}
\end{figure}

The main difference with the walk length $\tri x$ is that the sojourn time $\tri t$ is independent of the choice of the reference point.
In Figure \ref{fig4}, three different ways of taking a reference point are tested and the same steady states are obtained. For this test, we take a bounded domain $\Omega=(0,2\pi)$ with the periodic boundary condition and the sojourn time as
\begin{equation}\label{6.2}
\tau(x)=1+0.5\sin(x).
\end{equation}
The reference point for the spatial heterogeneity in $\tau$ is taken according to \eqref{reference_t}, i.e., when a particle jumps from $x$ to $y$,
\begin{equation}\label{6.3}
\tri t=\tau(by+(1-b)x),\quad b\in[0,1].
\end{equation}
For Monte Carlo simulations, a total number of $10^6$ particles are used. The initial distribution is taken with a random distribution as above. The walk length $\tri x$ follows the normal distribution with a deviation of $0.2$. The particle distributions at $t=5$ are shown in Figure \ref{fig4}. The actual number of walks are given by the same relations of \eqref{6.D} and \eqref{6.T}. We can see that the histogram converges to a distribution proportional to the sojourn time $\tau(x)$ in \eqref{6.2}. We tested three cases where $b=0,0.5$ and $1$. We can see that all of them give the same steady state which is simply $C\tau(x)$ where $C$ is given by \eqref{6.C}. In fact, since the graph is a histogram with a bin size of $0.2$, we made the sine function graph in Figure \ref{fig4} with $C$ after multiplying the bin size. However, for the case of walk length and departing rate, the obtained steady states are all different depending on the choice of the reference points (see \cite[Section 7]{006}).

\begin{figure}[ht]
\small
\centering
 \includegraphics[width=0.32\textwidth]{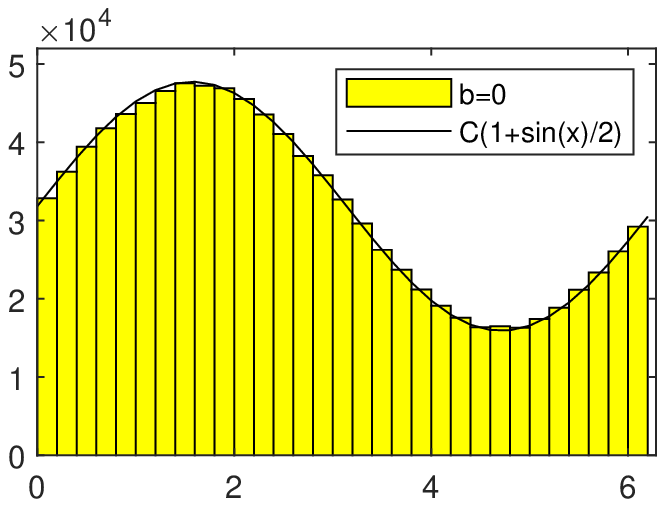}
 \includegraphics[width=0.32\textwidth]{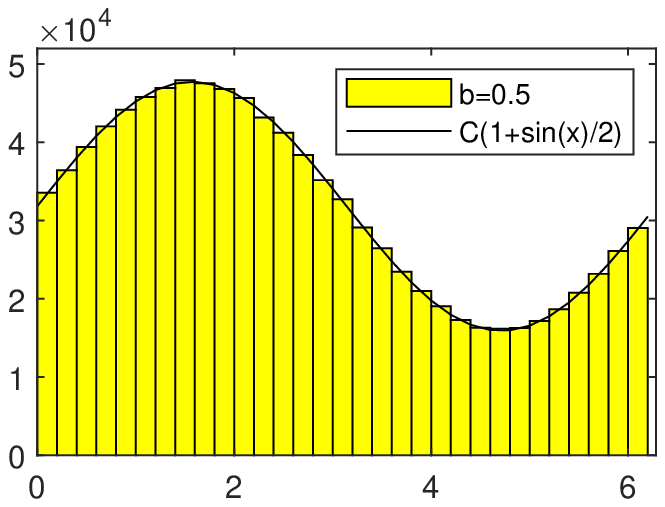}
 \includegraphics[width=0.32\textwidth]{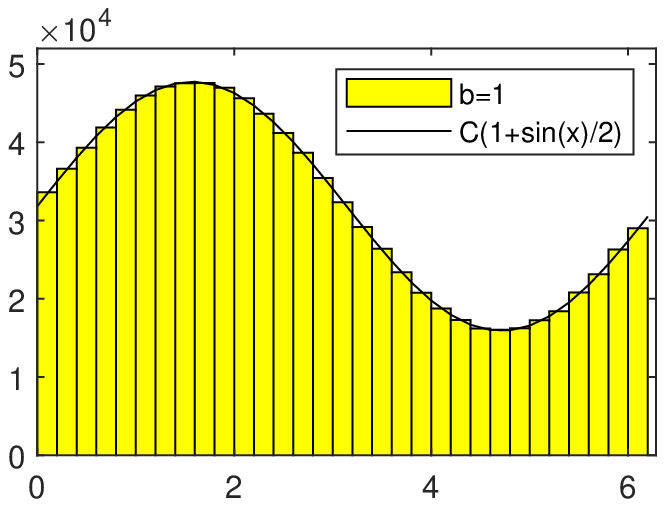}
\caption{\small{\bf Histogram and steady states.} The sojourn time $\tau(x)$ is given by \eqref{6.2} and $\tri t$ is taken as \eqref{6.3}. The steady states are computed. They are independent of the reference point.}\label{fig4}
\end{figure}

\subsection*{Acknowledgement} Jaywan Chung was supported by the National Research Foundation of Korea (No. 2021M3D1A2043845). Yong-Jung Kim was supported in part by the National Research Foundation of Korea (No. 2022R1H1A2092302). Min-Gi Lee was supported by the National Research Foundation of Korea(NRF) grant funded by the Korea government(MSIT) (No. 2020R1A4A1018190, 2021R1C1C1011867).

\appendix

\section{MATLAB codes for Monte Carlo simulations}

The codes for the three Monte Carlo simulations in the paper are given below. The simulations are performed using MATLAB.

\subsection{Code for Figure \ref{fig2}:}
{\small
\begin{verbatim}
%% parameters
N=1000000;      % number of particles
t=0.5;          % time for Green's function scale
dx=0.05;        % step size
%% variables
T=2*t/dx/dx;    % maximum number of steps DT=t, D=dx^2/2
X=zeros(1,N);   % Position of N particles
%% computations
for i=1:N
  t=T;
  while(t>0)
    X(i)=X(i)+randn()*dx;   % normal distribution with deviation dx
    t=t-(sign(X(i))+3)/2;   % sojourn time taken from final position
  end
end
%% display
x=-5:0.1:5;
histogram(X,x,'FaceColor','y');
legend('Monte-Carlo');
axis([-5 5 0 80000]);
\end{verbatim}
}

\subsection{Code for Figure \ref{fig3}:}
{\small
\begin{verbatim}
%% parameters
N=1000000;      % number of particles
L=4;            % domain is [0 L] with periodic BC
t=1;            % time for steady state
dx=.1;         % step size
%% variables
T=2*t/dx/dx;    % maximum number of steps DT=t, D=dx^2/2
X=rand(1,N)*L;  % Position of N particles
%% computations
for i=1:N
  t=T;
  while(t>0)
    X(i)=X(i)+randn()*dx;  % normal distribution with deviation dx
    X(i)=mod(X(i),L);      % periodic condition
    t=t-ceil(X(i));        % sojourn time taken from final position
  end
end
%% display
x=0:0.1:L;
histogram(X,x,'FaceColor','y','FaceAlpha',1);
legend('t=1'); axis([0 L 0 50000]);
\end{verbatim}
}

\subsection{Code for Figure \ref{fig4}:}
{\small
\begin{verbatim}
%% parameters
N=1000000;      % number of particles
L=2*pi;            % domain is [0 L] with periodic BC
t=5;            % time for steady state
dx=0.2;         % step size
b=1;
%% variables
T=2*t/dx/dx;    % maximum number of steps DT=t, D=dx^2/2
X=rand(1,N)*L;  % Position of N particles
%% computations
for i=1:N
    t=T;
    while(t>0)
        Y=X(i)+randn()*dx;
        t=t-(1+0.5*sin(b*Y+(1-b)*X(i)));
        X(i)=Y;
    end
end
X=mod(X,L);     % periodic boundary condition
%% display
x=0:dx:L;C=dx*N/L;
histogram(X,x,'FaceColor','y','FaceAlpha',1);hold on;
plot(x,C*(sin(x)/2+1),'-k');
legend('b=1','C(1+sin(x)/2)'); axis([0 L 0 26000]);
\end{verbatim}
}

\end{document}